\documentclass{amsart}
\usepackage{amsfonts,amssymb,amscd,amsmath,enumerate}

\usepackage[bottom]{footmisc}

\newtheorem{theorem}{Theorem}[section]
\newtheorem{proposition}[theorem]{Proposition}

\usepackage{xy}
\xyoption{all}

\def\A{{\mathbf A}}

\def\B{{\mathbf B}}
\def\C{{\mathbf C}}

\def\N{{\mathbf N}}

\def\Oo{{\mathcal O}}
\def\P{{\mathbf P}}

\def\R{\mathbf R}

\def\int{\mathrm{int}}

\providecommand{\bysame}{\leavevmode\hbox to3em{\hrulefill}\thinspace}

\begin{document}
\bibliographystyle{amsplain}
\setcounter{tocdepth}{2}
\title{Some ideas in need of clarification\\in resolution of singularities and the geometry of discriminants}
\author{Bernard Teissier}
\address{Universit\'e Paris Cit\'e and Sorbonne Universit\'e, CNRS, IMJ-PRG, F-75013 Paris, France}

\email{bernard.teissier@imj-prg.fr}


\pagestyle{myheadings}
\markboth{\rm Bernard Teissier}{\rm Ideas in need of clarification}
\renewcommand\rightmark{B. Teissier}
\renewcommand\leftmark{Ideas in need of clarification}

\maketitle
\textit{\small Mathematics only exists in a living community of mathematicians that spreads understanding and breathes life into ideas both old and new. The real satisfaction from mathematics is in learning from others and sharing with others. All of us have clear understanding of a few things and murky concepts of many more. There is no way to run out of ideas in need of clarification.}\par\medskip\noindent
\hfill \textit{W. P. Thurston, quoted in \cite{L-P}}\par\bigskip\noindent
 { \small Among the different ways of sharing ideas are discussions, letters, videos, blogs, platforms such as MathOverflow or Images des Math\'ematiques, and publications. Publications are subject to more precise rules, and require more prolonged effort because they are not only a means of communication but also the main repository of ideas and results.\par That is why the  activity of Catriona over four decades has been so useful for the mathematical community. She has a unique way of imagining possible publications, encouraging without pushing, showing great patience and understanding adapted to each author (or editor). She possesses an amazingly rich perception of the mathematical community, knowing of so many mathematicians not only what they do, but also what they are. Catriona really cares about authors (or editors) as persons as well as about the quality of the texts. Adding to this an inexhaustible energy, Catriona plays a unique and very important role in the spreading of understanding, as Thurston writes, and thus for the progress of our science. As an expression of gratitude and friendship, I wish to dedicate to her an exposition of some of the problems I have come across and so illustrate the last sentence of Thurston's quote. }\par\medskip\noindent
\begin{abstract}{This text presents problems in two different areas of algebraic geometry. The first concerns the role of ``infinitely singular" or ``non Abhyankar" valuations in the study of local uniformization of valuations with a view to resolution of singularities in positive characteristic. The second one concerns the relationship of the geometry of the discriminant of real miniversal unfoldings in the sense of Thom with the movements of Morse functions on a cobordism which differential geometers use. }
\end{abstract}
\section{Problems related to resolution of singularities}
In his 1964 paper Hironaka introduced the general concept of \textit{embedded resolution} of a singular space $X$ embedded in a non singular variety $Z$. That is a birational morphism $b\colon Z'\to Z$ with $Z'$ non singular, such that the strict transform of $X$ is non singular and transversal in $Z'$ to the exceptional divisor of $b$, which is mapped to the singular locus of $X$. Indeed, Hironaka's proof builds $Z'$ as the result of a sequence of blowing-ups with non singular centers. If $X$ is a toric variety equivariantly embedded in a non singular toric variety $Z$, it is more natural to seek a birational toric morphism $Z'\to Z$ of non singular toric varieties such that the strict transform of $X$ is non singular and transversal to the toric boundary of $Z'$. For toric varieties over an algebraically closed field, this was proved to exist in \cite[\S 6]{T5} and \cite{GP-T}. The process is purely combinatorial and therefore blind to the characteristic of the field.\par
If $\Oo$ is the local ring of a formal branch $C$ over an algebraically closed field $k$, its normalization is $k[[t]]$ and the set of values which the $t$-adic valuation $\nu$ takes on elements of $\Oo$ is a numerical semigroup $\Gamma\subset \N$. It is finitely generated. For tradition's sake, we denote by $g+1$ its minimal number of generators. The associated graded ring ${\rm gr}_\nu\Oo$ of the valuation $\nu$ restricted to $\Oo$ (see \cite[\S 2]{T5}, \cite{T1}) is isomorphic to the semigroup algebra $k[t^\Gamma]$  of $\Gamma$ with coefficients in $k$ and thus corresponds to an affine toric variety $C^\Gamma\subset \A^{g+1}(k)$. If $\xi_0,\xi_1,\ldots ,\xi_g$ are elements of $\Oo$ whose $\nu$-values generate $\Gamma$, the image of the formal embedding of $C$ in the affine space $\A^{g+1}(k)$ determined by the $\xi_i$ can be degenerated to $C^\Gamma$ inside $\A^{g+1}(k)$ in such a way that some toric embedded resolutions of $C^\Gamma\subset \A^{g+1}(k)$ also give embedded resolutions of $C\subset \A^{g+1}(k)$. All this is blind to the characteristic of the field $k$. An instance of this, in the complex analytic world, first appeared in \cite{G-T} and recently the case of reduced plane curve singularities has been settled (and more) in \cite{dF-GP-M} and also in \cite[Corollary 7.11]{C-PP-S}.\par\medskip
An attempt to generalize this leads to the following, where $k$ is an algebraically closed field.\par\medskip\noindent
\textbf{Problem A:} \textit{Let $X\subset \A^n(k)$ be a reduced affine algebraic variety over $k$. Do there exist algebraic embeddings $\A^n(k)\subset \A^N(k)$ such that:
\begin{enumerate} 
\item The intersection of the image of $X$ (resp. $\A^n(k)$) with the torus of $\A^N(k)$ is dense in $X$ (resp. $\A^n(k)$);
\item There exist birational equivariant maps $\pi \colon Z\to \A^N(k)$ of non singular toric varieties such that the strict transform $X^\pi$ of $X$ (resp. the strict transform of $\A^n(k)$), which exists by 1., is non singular and transversal to the toric boundary of $Z$.
\item The ideal in the ring $\Oo_X(X)$ of the singular subspace of $X$ is generated, up to integral closure, by monomials in the coordinates of $\A^N(k)$.
\end{enumerate}}
If the embedding $X\subset \A^n(k)\subset \A^N(k)$ satisfies the first two conditions, we call it a \textit{torific embedding} for $X$. For example, an isolated hypersurface singularity which is non degenerate with respect to its Newton polyhedron is torifically embedded in its ambient space. See \cite{Tev2} and \cite{A-GM-M} for generalizations.
Tevelev has shown in \cite{Tev1} that any embedded resolution diagram of irreducible projective varieties
\[\xymatrix{X'\ar[r]\ar[dd]^{\pi_X}& Z'\ar[dd]^\pi
           &  \\
             & &  \\
            X\ar[r]  & Z &}\]
can be embedded in a diagram:
\[\xymatrix{X'\ar[r]\ar[dd]^{\pi_X}& Z'\ar[r]\ar[dd]^\pi &W'\ar[dd]^\Pi
           &  \\
             & &  \\
            X\ar[r]  & Z \ar[r]&\P^N},\]
where $Z\to \P^N$ is an embedding, the map $\Pi$ is a birational toric map of non-singular varieties for a toric structure on $\P^N$, the images of $Z$ and $X$ have dense intersections with the torus of $\P^N$, their strict transforms are the non singular varieties $Z'$ and $X'$, and they are transversal to the toric boundary of $W'$.\par\noindent
In this sense, in characteristic zero where we have Hironaka's theorem, and more generally whenever embedded resolution can be proved, toric embedded resolutions are cofinal among embedded resolutions of a given irreducible projective variety $X$.  \par\medskip
\textit{Coming back to affine or local torific embeddings, the problem of course is to prove the existence of torific embeddings without assuming embedded resolution, in a way which hopefully would also work in positive characteristic.} As a bonus, torific embeddings should, as in the case of curves, contain important geometric information on the singularities of $X$, which do not seem, in dimension $\geq 3$, to be legible in the resolution by blowing ups. To my knowledge there are two approaches to this problem:\par\medskip\noindent
- Mourtada's approach (see \cite{Mo},  \cite{M-P}) is based on a deep vision of the relationship between components of the exceptional divisor (divisorial valuations centered in $Z$) of an embedded resolution of $X$ and contact subvarieties of the jet schemes on $Z$ associated to the embedding $X\subset Z$. Suitable irreducible components of the contact varieties mentioned above correspond to divisorial valuations centered in $\A^n(k)$ and the equations of each one give an embedding $\A^n(k)\subset\A^N(k)$ such that the divisorial valuation is the trace on $\A^n(k)$ of a \textit{monomial} divisorial valuation on $\A^N(k)$. Then a tropical/toroidal argument explains how to produce a torification. Anyway, that is the idea, and it proves extremely fruitful in spite of the complexity of the computation of the equations of irreducible components. Mourtada realizes this program in a number of important cases, which I shall not detail here.\par\medskip\noindent
- My approach is more directly inspired by the case of curves presented above. As in Zariski's approach, it begins with local uniformization of valuations. The reason is that if $\Oo_{X,x}$ is the local algebra of a singularity and $\nu$ is a rational\footnote{It means that the inclusion $R\subset R_\nu$ of $R$ in the valuation ring of the valuation not only satisfies $m_\nu\cap R=m_R$ for maximal ideals, but also there is no residual extension: $k=R/m_R=k_\nu =R_\nu/m_\nu$.} valuation centered in $\Oo_{X,x}$, the associated graded algebra ${\rm gr}_\nu\Oo_{X,x}$ is again isomorphic to the semigroup algebra $k[t^\Gamma]$ of the semigroup $\Gamma$ of values taken by $\nu$ on $\Oo_{X,x}$.\par\noindent If $\Gamma$ is finitely generated, then we have again an affine toric variety and we can show that toric resolutions of this toric variety, which are blind to the characteristic, provide local uniformization of the valuation after a suitable re-embedding of $(X,x)$ (see \cite{T1}).\par However, the semigroup $\Gamma$ is not at all finitely generated in general and we have to think of ${\rm Spec}k[t^\Gamma]$ as being of infinite embedding dimension, this embedding dimension being in fact an ordinal, see \cite[corollary 3.10]{T1}. Such a toric variety is defined by an infinite collection of binomials and does not have a resolution so we have to show that a \textit{``finite partial" embedded resolution} extends to a local uniformization of $\nu$. In order to do that we need equations for the degeneration of $\Oo_{X,x}$ to its graded algebra. This is something we can do for \textit{complete} equicharacteristic noetherian local domains.\par\noindent
Indeed, if the noetherian equicharacteristic local domain $R$ is complete, there exists for any rational valuation of $R$ an embedding of the formal space corresponding to $R$ into the space where the generalized toric variety corresponding to ${\rm gr}_\nu R$ resides; it is given by the Valuative Cohen Theorem of \cite{T1}.\par\noindent Since $R$ is noetherian, the semigroup $\Gamma$ is well ordered and combinatorially finite in the sense that there are finitely many distinct expressions of an element of $\Gamma$ as a sum of other elements. As a consequence of being well ordered it has a unique minimal system of generators $(\gamma_i)_{i\in I}$, the index set $I$ being an ordinal $\leq\omega^{h(\nu)}$ where $h(\nu)$ is the height, or rank, of the valuation, which is $\leq{\rm dim}R$. Taking variables $(u_i)_{i\in I}$, one can consider the $k$-vector space of all formal sums $\Sigma_{e\in E}d_eu^e$ where $E$ is any set of monomials in the $u_i$ and $d_e\in k$.  Since the values semigroup $\Gamma$ is well ordered and combinatorially finite this vector space is in fact, with the usual multiplication rule, a $k$-algebra $\widehat{k[(u_i)_{i\in I}]}$ which we endow with a weight $w$ by giving $u_i$ the weight $\gamma_i$. Combinatorial finiteness means that there are only finitely many monomials with a given weight, and we can enumerate them according to the lexicographic order of exponents (see \cite[\S 4]{T1}). Thus, we can embed the set of monomials $u^m$ in the well ordered lexicographic product $\Gamma\times \N$. Combinatorial finiteness also implies that the initial form of every series with respect to the weight filtration is a polynomial so that the corresponding graded algebra of $\widehat{k[(u_i)_{i\in I}]}$ is the polynomial algebra $k[(U_i)_{i\in I}]$ with $U_i={\rm in}_wu_i$, graded by giving $U_i$ the degree $\gamma_i$.\par
The $k$-algebra $\widehat{k[(u_i)_{i\in I}]}$ is endowed with a monomial valuation given by the weight: $w(\Sigma_{e\in E}d_eu^e)={\rm min}_{d_e\neq 0}w(u^e) $. This valuation is rational since all the $\gamma_i$ are $>0$. Note that $0$ is the only element with value $\infty$ because here $\infty$ is an element larger than any element of $\Gamma$. With respect to the "$w$ ultrametric" given by $u(x,y)=w(y-x)$, the algebra $\widehat{k[(u_i)_{i\in I}]}$ is spherically complete (see \cite{T3}, theorem 4.2) and has most of the properties of power series algebras, except for noetherianity unless the set $I$ is finite, in which case it is isomorphic to the usual power series ring, with weights on the variables.\par\noindent
The $\gamma_i$ are the degrees of a minimal set of homogeneous generators $(\overline{\xi_i})_{i\in I}$ of the $\Gamma$-graded $k$-algebra ${\rm gr}_\nu R$. The first part of the valuative Cohen theorem asserts that one can choose representatives $(\xi_i)_{i\in I}$ in $R$ of the $(\overline{\xi_i})_{i\in I}$ in such a way that $u_i\mapsto\xi_i $ determines a surjective continuous (with respect to the valuations) map of $k$-algebras
$$\pi\colon \widehat{k[(u_i)_{i\in I}] }\rightarrow R$$
whose associated graded map with respect to the filtrations associated to the valuations is the surjective graded map of $k$-algebras
$${\rm gr}\pi\colon k[(U_i)_{i\in I}]\rightarrow {\rm gr}_\nu R,\ \ U_i\mapsto \overline{\xi_i}.$$
If the valuation $\nu$ is of rank one or the set $I$ is finite, any set of representatives $(\xi_i)_{i\in I}$ of the $(\overline{\xi_i})_{i\in I}$ is eligible.\par\noindent
Since even when the set $I$ is infinite the non zero homogeneous components of ${\rm gr}_\nu R$ are one dimensional $k$-vector spaces, and in fact ${\rm gr}_\nu R$ is isomorphic to the semigroup algebra $k[t^\Gamma]$ (see \cite[Proposition 4.7]{T5}), the kernel of the map ${\rm gr}\pi$ is a prime ideal generated by binomials $(U^m-\lambda_{mn}U^n)_{(m,n)\in M}$ with $\lambda_{mn}\in k^*$, and the second part of the valuative Cohen theorem states that the kernel of $\pi$ is generated, up to closure in the $w$ ultrametric, by \textit{overweight deformations} of those binomials: series whose initial forms with respect to the weight are those binomials.\par\noindent
Geometrically this corresponds to equations defining the image of an embedding, via the series $\xi_i$, of the formal germ corresponding to $R$ in an infinite dimensional weighted affine space in such a way that the original valuation is the trace on $R$ of a monomial valuation on the ambient space.\par\noindent
In this infinite dimensional space, our singularity can be degenerated in a faithfully flat way to the ``toric variety'' defined by the binomials $U^m-\lambda_{mn}U^n$ of $k[(U_i)_{i\in I}]$ (see \cite[Proposition 2.3]{T5}). However, if the number of variables is infinite, there is no resolution of singularities for such a generalized toric variety. It is truly ``infinitely singular" and in fact for valuations centered in $k[[x,y]]$ this corresponds exactly to the ``infinitely singular" case where the valuation is the order of vanishing of a series on a very transcendental (non Puiseux) curve in the plane whose strict transforms remains singular in infinitely many point blowing ups (see \cite[Examples 4.20 and 4.22]{T5}).\par On the other hand, we know that for rational Abhyankar valuations (= of rational rank equal to the dimension of $R$), one can prove that  after a birational modification of $R$ to an $R'$ still dominated by the valuation ring of $\nu$, we can obtain that the semigroup of values of the valuation $\nu$ on $R'$ is finitely generated. For this reason, rational Abhyankar valuations on an equicharacteristic excellent local domain with an algebraically closed residue field can be uniformized (see \cite{T1}, and \cite{K-K} and \cite{C2} for different approaches for algebraic function fields). This leads to the following conjecture for non-Abhyankar valuations, whose semigroup cannot be finitely generated (see also \cite{T6}):\par\medskip\noindent
\textit{Let $R$ be a complete equicharacteristic local domain with algebraically closed residue field and $\nu$ a rational valuation centered in $R$ and of rational rank $r<{\rm dim}R$. Let $(\gamma_i)_{i\in I}$ be the minimal system of generators for the semigroup $\Gamma$ of $\nu$ on $R$. There exist a nested system of finite subsets $B_\alpha\subset I$ with $\bigcup_\alpha B_\alpha=I$ and for each $B_\alpha$ a prime ideal $K_\alpha$ of $R$ such that $R/K_\alpha$ is of dimension $r$ and endowed with an Abhyankar valuation $\nu_\alpha$ whose semigroup is generated by the $(\gamma_i)_{i\in B_\alpha}$. We have $\bigcap_\alpha K_\alpha=(0)$ and for each $x\in R$ we have $\nu(x)=\nu_\alpha (x\ {\rm mod.}K_\alpha)$ for large enough $B_\alpha$. Finally each $R/K_\alpha$ is an overweight deformation of an affine toric variety and for large enough $B_\alpha$ toric embedded uniformizations of the valuation $\nu_\alpha$ also uniformize the valuation $\nu$ on $R$.}\par\medskip
It is a convenient way to express that the valuation can be uniformized by ``finite partial'' embedded toric resolutions of ${\rm Spec}k[t^\Gamma]$, an adapted form of torific embedding for the valuation. The slogan is: \textit{Approximating a rational non-Abhyankar valuation $\nu$ by rational Abhyankar semivaluations\footnote{A semivaluation of $R$ is a valuation of a quotient of $R$ by a prime ideal.} should provide torific embeddings for $\nu$}.\par
 In order to apply this to our algebraic situation we have to deduce a torific embedding for an algebraic local ring from a torific embedding of a complete local ring to which we can apply the valuative Cohen theorem. For that it suffices to solve the following problem (see \cite[*Proposition 5.19*]{T5} and \cite{HOST}):\par\medskip\noindent
\textbf{Problem B:} \textit{Given an excellent equicharacteristic local domain $R$ and a valuation $\nu$ centered in $R$, show that there exists a prime ideal $H$ of the $m_R$-adic completion $\hat R$ such that $H\cap R=(0)$ and $\nu$ extends to a  valuation $\hat\nu$ of $\hat R/H$ with the same value group.}\par\noindent
This means that the graded inclusion ${\rm gr}_\nu R\subset {\rm gr}_{\hat\nu}\hat R/H$ is birational.\par\medskip If the valuation is of rank one, the proof is in \cite[Theorem 2.1]{HOST}. For Abhyankar valuations, the proof is in \cite[7.2]{T1}.\par\noindent Conjecture 9.1 in \cite{HOST} to the effect that after a birational $\nu$-modification $R'$ of $R$ one can even have that the semigroup of $\hat R'/H'$ is the same as that of $R'$ has been disproved by Cutkosky in \cite[Theorem 1.5, Theorem 1.6]{C} even when completion is replaced by henselization.\par\medskip\noindent
Since we assume that the residue field $k$ is algebraically closed, local uniformization of rational valuations entails local uniformization for all valuations (see \cite[Proposition 3.20]{T5}). By (quasi-)compactness of the Riemann-Zariski manifold, the space of valuations centered in $R$ is (quasi-)compact and therefore there are finitely many valuations such that the collections of morphisms uniformizing them by toric embedded uniformizations uniformizes all valuations centered in $R$. Now the problem is to:\par\medskip\noindent
\textbf{Problem C:} \textit{Prove that those torific embeddings can be combined into one embedding for ${\rm Spec}R$ where a toric birational map will simultaneously uniformize all valuations and thus provide a local embedded resolution of singularities for $R$.}\par\noindent One can find some inspiration in \cite[\S 3]{dF-GP-M} as well as in the local tropicalization methods of \cite{PP-S} and \cite{C-PP-S}. 

\section{Problems related to the geometry of discriminants of miniversal unfoldings}
Let $f(z_1,\ldots ,z_n)\in \R\{z_1,\ldots ,z_n\}$ be a series without constant term and such that it has an algebraically isolated critical point at the origin, which means that ${\rm dim}_\R\R\{z_1,\ldots ,z_n\}/(\frac{\partial f}{\partial z_1},\ldots ,\frac{\partial f}{\partial z_n})<\infty$. This dimension is the Milnor number $\mu$ of the isolated critical point associated to the complexification of the series $f(z_1,\ldots ,z_n)$ (see \cite{Mi2}). A function with an algebraically isolated critical point is finitely determined, so we may assume that $f$ is a polynomial. Let us consider an unfolding of the function $f$, say
$$F(z,t)=f(z_1,\ldots ,z_n)+\sum_{k=1}^{\mu-1}t_kg_k(z_1,\ldots ,z_n),$$
which is miniversal (see \cite[Chap. 8]{A-GZ-V}) if the images of the functions $1,g_1,\ldots ,g_{\mu-1}$, which again we may take to be polynomials, even monomials, form a basis of the real vector space $\R\{z_1,\ldots ,z_n \}/(\frac{\partial f}{\partial z_1},\ldots ,\frac{\partial f}{\partial z_n})$, which we shall henceforth assume.\par\noindent
This unfolding defines a germ of a stable map (see \cite[Chap. III, Theorem 3.4]{G-W-dP-L} or \cite[Chap 5]{M-NB})$${\mathbf F}=(F,t) \colon (\R^n\times\R^{\mu-1},0)\rightarrow (\R\times\R^{\mu-1},0)$$  
expressed in the natural coordinates $(\lambda,t)=(\lambda, t_1,\ldots, t_{\mu-1})$ on $\R\times \R^{\mu-1}$  by: \begin{align*} \lambda\circ {\mathbf F}=&f(z_1,\ldots ,z_n)+\sum_{k=1}^{\mu-1}t_kg_k(z_1,\ldots ,z_n),\cr t_k\circ {\mathbf F}=&t_k\ \ {\rm for}\ k=1,\ldots ,\mu-1.\end{align*}
Because ${\mathbf F}$ is a stable map, it can be Thom-stratified (see \cite{M}, \cite{G-W-dP-L}) and there exist ``polycylinders'' $U=\B_n\times \B_{\mu-1}\subset \R^n\times\R^{\mu-1}$ and $V=\B_1\times\B_{\mu-1}\subset\R\times\R^{\mu-1}$ such that $F^{-1}(V)\cap U$ is a neighborhood of $0$ in which the critical locus $C$ is non singular, and $C\cap(\partial \B_n\times \B_{\mu-1})=\emptyset$. The only critical points which appear in that neighborhood are those which tend to $0$ as $t\to 0$, and each fiber ${\mathbf F}^{-1}(\lambda, t)$ for $(\lambda, t)\in V$ is transversal to $\partial \B_n\times \{t\}$. We shall freely assume that the closed balls $\B_e\subset \R^e$ are ``small enough''.\par
Now recall that, assuming of course that $0$ is a critical point,  up to a change of variables, our function can be written as $$f(z_1,\ldots ,z_n)=\sum_{j=1}^{q^+}z_i^2-\sum_{j=q^++1}^{q^++q^-}z_j^2+\tilde f(z_{q^++q^-+1},\ldots ,z_n),$$ where $\tilde f$ is of order $\geq 3$ and has the same Milnor number as $f$. \par\noindent
Then the algebra $\R\{z_1,\ldots ,z_n\}/(\frac{\partial f}{\partial z_1},\ldots ,\frac{\partial f}{\partial z_n})$ is naturally isomorphic to \break$\R\{z_{q^++q^-+1},\ldots ,z_n\}/(\frac{\partial \tilde f}{\partial z_{q^++q^-+1}},\ldots ,\frac{\partial \tilde f}{\partial z_n})$. A miniversal unfolding $\tilde F$ of $\tilde f$ is miniversal for $f$, the only difference between $F$ and $\tilde F$ being a fixed difference of indices between the Morse singularities appearing in the unfoldings. \par
From now on we shall assume that the order of $f(z_1,\ldots ,z_n)$ is $\geq 3$. Then we may choose $g_i=z_i$ for $i=1,\ldots, n$ and $g_k$ of order $\geq 2$ for $k>n$. The equations for the critical locus $C\subset \B_n\times\B_{\mu-1}$ of the unfolding $F$ are $$\frac{\partial F}{\partial z_i}=\frac{\partial f}{\partial z_i}+t_i+\sum_{k=n+1}^{\mu-1}t_k\frac{\partial g_k}{\partial z_i}=0\ {\rm for}\ i=1,\ldots ,n,$$
showing that $C$ is non singular and of dimension $\mu-1$.\par Shrinking the balls $\B_n,\ \B_{\mu-1}$ if necessary, we assume that the map \break$\nu\colon C\to \B\times\B_{\mu-1}$ induced by ${\mathbf F}$ is finite by the Weierstrass preparation theorem (an analytic map with a finite fiber is locally finite). Its image is the real part $D$ of a complex hypersurface, the discriminant $D(\C)$ of the complexification of the morphism ${\mathbf F}$ (see \cite[\S 5]{T2}). We have seen that $C$ is non singular, and the map $\nu\colon C\to D$ is finite. As image of $C$, and because all maps in sight are algebraic, the discriminant $D$ is a semialgebraic hypersurface in $\B_1\times\B_{\mu-1}$. \par
A miniversal unfolding of an algebraically isolated singularity of hypersurface is a versal deformation so that we can lift to $D$ the properties of discriminants of miniversal deformations, and in particular the product decomposition theorem of \cite[Chap. III, Th\'eor\`eme 2.1]{T7} and \cite[Theorem 4.8.1, Cor. 4.8.2]{T2} which remains true in real geometry and implies that non singular points of $D$ are the images of Morse singularities in $C$. It also implies that at a general point of the codimension one components of the singular locus of  the complexification $D(\C)$, the singular locus is locally isomorphic either to a cusp ($y^2-x^3=0$) times $\C^{\mu-2}$ (cusp type) or to a node ($y^2\pm x^2=0)$ times $\C^{\mu-2}$ (node type).
\begin{proposition}
The zero set in the critical locus $C$ of the hessian determinant $h_z(F)$ of $F$ with respect to the variables $z_1,\ldots, z_n$ is of codimension one. 
\end{proposition}
\begin{proof} In the real space or in the complexification the image of the zero set of the hessian is the part of the singular locus which is the closure of the set of points of cusp type. The real part of a complex point of cusp type is a real point of cusp type, locally isomorphic to a cusp times $\R^{\mu-2}$. \end{proof}
 The singular locus of $D$ is of codimension one and its image $\Delta$ in $\B_{\mu-1}$ is a semialgebraic hypersurface containing the bifurcation locus $\Sigma$ and the conflict strata in the sense of bifurcation theory. Indeed, a point $t=(t_1,\ldots ,t_{\mu-1})$ is in $\B_{\mu-1}\setminus \Delta$ if and only if the corresponding function $F_t=F(z,t)\colon\R^n\to \R$ is an excellent\footnote{Meaning Morse function with distinct critical values.} Morse function in $\B_n$, all of whose Morse singularities tend to $0$ as $t\to 0$. In particular the Maxwell set, which corresponds to functions $F_t$ attaining at least twice their absolute minimum, is contained in $\Delta$ because it is the image of a singular stratum of $D$ (see \cite[5.4.1]{T2} and Michel Coste's examples in \cite{Co}).\par\medskip
We know that the geometry of the complex discriminant hypersurface $D(\C)$ contains important information on the geometry of the hypersurface of $\C^n$  defined by $f(z_1,\ldots ,z_n)=0$, its deformations and in particular its Milnor fiber (see \cite{T2}). The geometry of the discriminant hypersurface is very special. For example, tangent hyperplanes to the discriminant hypersurface at non singular points tending to the origin all have as limit the hyperplane $\lambda=0$ (see \cite[\S 5, Remark 3]{T2}) and as we have seen a general plane section of $D(\C)$ has only cusps and nodes as singularities (see \cite[4.8.2]{T2}).\par
The geometry of the discriminant $D$ in the real case also contains important information. I would like to state two problems concerning this geometry:\par\medskip\noindent
Given $f(z_1,\ldots ,z_n)\in \R[z_1,\ldots ,z_n]$ as above, Michel Herman asked, in the early 1990's, the following question: \par\medskip\noindent \textit{If in a neighborhood of $0$ the family of hypersurfaces $f(z_1,\ldots ,z_n)=\lambda$ is topologically trivial for $\vert\lambda\vert$ small enough, do there exist a neighborhood $U$ of $0$ and an unfolding $f(z_1,\ldots ,z_n)+sg(s,z_1,\ldots ,z_n)$ such that for $s\neq 0$ and small enough, the function $f(z_1,\ldots ,z_n)+sg(s,z_1,\ldots ,z_n)$ has no critical point in $U$?}\par\medskip
The geometric translation of this statement is that under the hypothesis of topological triviality, for a suitable representative of the germ \break${\mathbf F}\colon (\R^n\times \R^{\mu-1},0)\to (\R\times \R^{\mu-1},0)$, the map $p\circ\nu\colon C\to \B_{\mu-1}$ which we have seen above is not surjective. Indeed, if that is the case the complement of the image of $C$ being semialgebraic we can find (see \cite[Theorem 2.2.5]{B-C-R}) in that complement a germ of a semialgebraic arc $t_1(s),\ldots, t_{\mu-1}(s)$ with $t_j(0)=0$, which will give the unfolding we seek. The converse follows from the versality of the unfolding.\par\medskip\noindent
From the equations of the critical locus, we see that it can be endowed with coordinates $z_1,\ldots ,z_n, t_{n+1},\ldots ,t_{\mu-1}$ and then the map $p\circ \nu\colon C\to \R^{\mu-1}$ can be written as follows:
\begin{align*} t_j\circ(p\circ \nu)=&-\Bigl(\frac{\partial f}{\partial z_j}(z_1,\ldots ,z_n)+\sum_{k=n+1}^{\mu-1}t_k\frac{\partial g_k}{\partial z_j}(z_1,\ldots ,z_n)\Bigr)\ {\rm for}\ 1\leq j\leq n,\cr
t_j\circ(p\circ \nu)=&\ t_j\ {\rm for}\ \ n+1\leq j\leq \mu-1.\cr
\end{align*}
The jacobian matrix of the map $p\circ \nu$ is therefore related to the hessian matrix $H_z(F)$ of $F$ with respect to the variables $z_1,\ldots ,z_n$ as follows:
\[{\rm Jac}(p\circ \nu)=\left(\begin{array}{cc}-H_z(F)&(-\frac{\partial g_k}{\partial z_j})\\
\mathbf{0}&\mathbf{Id}_{\mu-1-n}\end{array}\right)\]
Taking determinants gives :
$${\rm jac}(p\circ\nu)=(-1)^nh_z(F),$$
and considering signs gives, at each point of $C$ where ${\rm jac}(p\circ\nu)\neq 0$,
$${\rm sign}({\rm jac}(p\circ\nu))=(-1)^n(-1)^{{\rm index}H_z(F)}.$$
For $t\in\R^{\mu-1}\setminus \Delta$, let $N_i(t)$ be the number of critical points of index $i$ of the Morse function $F_t$ on $\B_n$. Then by definition of the local topological degree (see \cite{E-L}), we have the equality

$${\rm deg}(p\circ\nu)=(-1)^n\sum_{i=0}^n(-1)^iN_i(t),$$
which is independent of $t\in\B_{\mu-1}\setminus\Delta$.\par
As $t$ approaches the origin, the discriminant $D$ flattens towards the hyperplane $\lambda=0$ (see \cite[5.5]{T2}). We shall only use the fact that for $t\in p(D)$ the line $(\lambda, t),\lambda\in\R$, has a maximum intersection point $(\lambda_{\rm max}(t),t)$ and a minimum intersection point $(\lambda_{\rm min}(t),t)$ with $D$, which both tend to $0$ with $t$, and $D$ does not contain the $\lambda$ axis. If we denote by $X_{\lambda, t}$ the fiber ${\mathbf F}^{-1}(\lambda,t)\subset \B_n\times\{t\}$, we note that since we may assume the $X_{\lambda, t}$ meet $\partial\B_n\times\{t\}$ transversally, all the fibers $X_{\lambda, t}$ for $\lambda>\lambda_{\rm max}(t)$ (resp $\lambda<\lambda_{\rm min}(t)$) are diffeomorphic to $X_{\lambda, 0}$ with $\lambda>0$ (resp $X_{\lambda, 0}$ with $\lambda<0$).\par\noindent  By a direct application of Morse theory (see \cite[Chapitre 13, exerc. 2.12]{G} and \cite[Lemma]{A}), we have for small enough $\epsilon>0$ the following relations between Euler-Poincar\'e characteristics  :
\[\begin{array}{ccc}\chi (X_{\lambda_{{\rm max}}+\epsilon,t})-\chi (X_{\lambda_{{\rm min}}-\epsilon,t})=&2\sum_{i=0}^n(-1)^iN_i(t)&{\rm if}\ n \ is \ {\rm odd}\\
\chi (X_{\lambda_{{\rm max}}+\epsilon,t})-\chi (X_{\lambda_{{\rm min}}-\epsilon,t})=&0\ \ \ \ \ 
\ \ \ \  \ \ \ \ \ \ \ \ \ \ \ &\ {\rm if}\ n \ is \ {\rm even}.\end{array}\]
One can verify that if the family $f(z_1,\ldots ,z_n)=\lambda$ is topologically trivial for $\vert\lambda\vert$ small enough, so is the family $f(z_1,\ldots ,z_n)+w^2=\lambda$ and as we saw, adding squares of new variables does not change the geometry of the miniversal unfolding.\par\noindent Topological triviality of the $X_{\lambda, 0}$ implies $\chi (X_{\lambda_{{\rm max}}+\epsilon,t})-\chi (X_{\lambda_{{\rm min}}-\epsilon,t})=0$, so that we have:
\begin{proposition}For all $n$ the hypothesis of local topological triviality of the family $f(z_1,\ldots ,z_n)=\lambda$ implies that the local topological degree of the map $p\circ\nu\colon C\to \B_{\mu-1}$ is zero.
\end{proposition}
And what we want to prove is that this map is not surjective.\par\medskip\noindent
When $n=2$ the result was proved by Gusein-Zade in \cite{GZ} using an ingenious argument to construct explicit unfoldings without critical points using induction on the Milnor number and resolution of singularities of curves.\par\noindent
In \cite{T4} it was suggested to use elimination of critical points as in the proof of the $h$-cobordism theorem (see \cite{Mi}, \cite{C-G}). In other words, is the condition $\sum_{i=0}^n(-1)^iN_i(t)=0$ sufficient to make it possible to eliminate all the critical points of a Morse function $F_t$ by moving $t$ in $\B_{\mu-1}$?  More generally, we wish to ask the question:\par\medskip\noindent
\textbf {Problem D:} \textit{Does the geometry of the discriminant $D$ reflect the various configurations of critical points of Morse functions which can appear in differential geometry: For example, can one find values of $t\in \B_{\mu-1}$ such that all the critical points of the same index of $F_t$ are at the same level  (have the same critical value)? Can one describe the obstruction to performing elimination of critical points of the functions $F_t$ by movements of $t\in \B_{\mu-1}$?}
\par\medskip\noindent
For example, it is explained in \cite{T3} that if one can find values of $t\in \B_{\mu-1}$ such that all non degenerate local minima of $F_t$ (stable attractors) are at the same level, one obtains a proof of Thom's catastrophe-theoretic version of the Gibbs phase rule. It states that the maximum number of local minima which a Morse function $F_t$ can have in $\B_n$ (the number of coexisting phases of the system) is at most the codimension in $\R^{\mu-1}$ of the Thom stratum $T_0$ of the origin, plus one. Since along the Thom stratum the morphology does not vary, the coordinates of a space transversal to $T_0$ and of complementary dimension in $\B_{\mu-1}$ are the essential parameters of the system.\par
There are some indications towards a possible proof in \cite{T4}. Since the hessian matrix of $F$ is a well defined matrix valued and very special function on the critical locus $C$, it may be that the study of its image in the Grothendieck-Witt ring, as in \cite{P-W}, is useful. Geometrically, the question is whether the closures of all the open sets of the discriminant hypersurface $D$ which are the images by $\nu$ of the sets of points of $C$ where the index of $H_z(F)$ has a given value, have to intersect and whether the closures of some other sets with different indices have to meet in codimension one cusp-type components of the singular locus of $D$. Two dimensional slices of the discriminant transversal to such components correspond to plane configurations studied by Jean Cerf (Cerf diagrams, see \cite{Ce}).\par\noindent From this viewpoint, we are interested in a dynamical version of Problem \textbf{D}: \textit{the problem is to understand which among the deformations of functions that are used in differential geometry, for example those used in the $h$-cobordism theorem (see \cite{Mi} and \cite{C-G}), can be realized by the ``small'' movements of $t$ in $\B_{\mu-1}$.}


\providecommand{\bysame}{\leavevmode\hbox to3em{\hrulefill}\thinspace}
\providecommand{\MR}{\relax\ifhmode\unskip\space\fi MR }
\providecommand{\MRhref}[2]{%
  \href{http://www.ams.org/mathscinet-getitem?mr=#1}{#2}
}
\providecommand{\href}[2]{#2}
\begin{thebibliography}{}

\end{thebibliography}


\begin{thebibliography}{A}
\bibitem{A} V. I. Arnol'd, \textit{Index of a singular point of a vector field, the Petrovski-Oleinik inequality and mixed Hodge structures}, Funkt. Anal. i. Priloj. \textbf{12}, 1, (1978), 1-11.
\bibitem{A-GZ-V} V. I. Arnol'd, S. M. Gusein-Zade and A. N. Varchenko, \textit{Singularities of differentiable maps, Vol. I}, Monographs in Math., Birkh\"auser,  Basel 1985.
\bibitem{A-GM-M} F. Aroca, M. G\'omez-Morales and H. Mourtada, \textit{Groebner Fan and embedded resolution of ideals on toric varieties},  arXiv:2202.10874 , 2022.
\bibitem{B-C-R} J. Bochnak, M. Coste and M.F. Roy, \textit{Real Algebraic Geometry}. Ergebnisse der Math., Folge 3, Vol. 36, Springer 1998.
 \bibitem{C} S. D. Cutkosky, \textit{Extensions of valuations to the Henselization and Completion}, Acta Mathematica Vietnamica 44 (2019), 159-172.
 \bibitem{C2} \bysame \textit{Local uniformization of Abhyankar valuations}, to appear in Mich. Math. J. Available at Michigan Math. J. Advance Publication, 1-33, (2021) DOI: 10.1307/mmj/20205888
 \bibitem{Ce} J. Cerf, \textit{La stratification naturelle des espaces de fonctions diff\'erentiables r\'eelles et le th\'eor\`eme de la pseudo-isotopie}, Publications Math\'ematiques de l'IHES, 39: 5-173 (1970).
 \bibitem{C-G} J. Cerf et A. Gramain, \textit{Le th\'eor\`eme du h-cobordisme (Smale)},
Cours profess\'e au printemps 1966 \`a la Facult\'e des Sciences d’Orsay.
Secr\'etariat Math\'ematique de l'Ecole Normale Sup\'erieure, 1968. 
Available at https://www.maths.ed.ac.uk/~v1ranick/surgery/cerfgram.pdf
\bibitem{Co}M. Coste in \textit{Images des Math\'ematiques}:  https://images.math.cnrs.fr/Une-sorciere-trois-parapluies-un-poisson.html
\bibitem{C-PP-S} M. A. Cueto, P. Popescu Pampu, and D. Stepanov, \textit{Local tropicalization of Splice Type Surface Singularities}, Preprint,   https://arxiv.org/pdf/2108.05912.pdf
 \bibitem{C-M-T} S. D. Cutkosky, H. Mourtada, and B. Teissier, \textit{On the construction of valuations and generating sequences on hypersurface singularities},  Algebraic Geometry 8 (6) (2021) 705-748. doi:10.14231/AG-2021-022
 \bibitem{E-L} D. Eisenbud and H.I. Levine, \textit{An algebraic formula for the degree of a $C\sp{\infty }$ map germ}, 
Ann. Math. (2) 106 (1977), no. 1, 19--44.
 \bibitem{dF-GP-M} A. de Felipe, P. Gonz\'alez P\'erez, and H. Mourtada, \textit{Resolving singularities of curves with one toric morphism}, arXiv:2110.11276v1 [math.AG] 21/10/2021.
 \bibitem{G-W-dP-L} C. G. Gibson, K. Wirthm\"uller, A.A. du Plessis, and E.J.N. Looienga, \textit{Topological stability of smooth mappings}. Springer Lecture Notes in Mathematics, no. 552, 1976.
 \bibitem{G} C. Godbillon, \textit{El\'ements de topologie alg\'ebrique}, Hermann, Collection M\'ethodes, Paris  1971.
  \bibitem{G-T} R. Goldin and B. Teissier, \textit{Resolving singularities of plane analytic branches with one toric morphism}, in:  Resolution of singularities (Obergurgl, 1997), 315-340, Progr. Math., 181, Birkhauser, Basel, 2000.
 \bibitem{GP-T} P. D. Gonz\'alez P\'erez and B. Teissier,  \textit{Embedded resolutions of non necessarily normal affine toric varieties}, C. R. Math. Acad. Sci. Paris, 334 (2002), no. 5.
 \bibitem{GZ} S. M. Gusein-Zade, \textit{On the existence of deformations without critical points (the Teissier problem for functions of two variables)}, Functional Analysis and Its Applications, Vol. 31, No. I, 1997, 58-60.
 \bibitem{HOST} F. J. Herrera Govantes, M. A. Olalla Acosta, M. Spivakovsky and B. Teissier, \textit{Extending a valuation centered in a local domain to the formal completion}, Proc. London Math. Soc. (2012) 105(3): 571-621. doi:10.1112/plms/pds002
 \bibitem{K-K} H. Knaf,  F.-V. Kuhlmann, Abhyankar places admit local uniformization in any characteristic. \textit{ Ann. Sci. \'Ecole Norm. Sup.} \textbf{38}  (4)  (2005),   833–846.
  \bibitem{L-P} F. Laudenbach and A. Papadopoulos, \textit{W. P. Thurston and French Mathematics}, EMS Surveys in Mathematical Sciences, Vol. 6, No. 1/2 (2019), 33-81. DOI 10.4171/EMSS/32
\bibitem{M} J. Mather, \textit{How to stratify mappings and jet spaces}, in : Plans sur Bex 1975, Springer Lecture Notes in Mathematics, no. 535 (1976), 128-176.
 \bibitem{Mi} J. Milnor, \textit{Lectures on the $h$-cobordism Theorem}, Princeton University Press, 1965.
 \bibitem{Mi2}\bysame, \textit{Singular points of complex hypersurfaces}, Annals of Math. Studies, no. 61 (1968), Princeton U.P.
 \bibitem{Mo} H. Mourtada, \textit{Jet schemes and generating sequences of divisorial valuations in
   dimension two}, Michigan Math. J., \textbf{66}, 1, (2017), 155-174.
   \bibitem{M-NB} D. Mond and J. Nu\~no-Ballesteros, \textit{Singularities of Mappings}, Grundlehren der mathematischen
Wissenschaften, vol. 357. Springer International Publishing (2020)
 \bibitem{M-P}H. Mourtada and C. Pl\'enat, \textit{Jet schemes and minimal toric embedded resolutions of rational
   double point singularities}, Comm. Algebra, \textbf{46}, 3, (2018), 1314-1332.
  \bibitem{P-W} S. Pauli and K. Wickelgren, \textit{Applications to $\A^1$-enumerative geometry of the $\A^1$-degree}, Research in Mathematical Sciences, \textbf{8}, 24 (2021).
 \bibitem{PP-S} P. Popescu Pampu and D. Stepanov, \textit{Local tropicalization}, in: Algebraic and
Combinatorial aspects of Tropical Geometry, Proceedings Castro Urdiales
2011, E. Brugall\'e, M.A. Cueto, A. Dickenstein, E.M. Feichtner and I.
Itenberg editors, Contemporary Mathematics 589, AMS, 2013, 253-316.
 \bibitem{T1} B. Teissier, \textit{Overweight deformations of affine toric varieties and local uniformization}, in "Valuation theory in interaction", Proceedings of the second international conference on valuation theory, Segovia-El Escorial, 2011. Edited by A. Campillo, F-V. Kuhlmann and B. Teissier. European Math. Soc. Publishing House, Congress Reports Series, Sept. 2014, 474-565. Available at https://webusers.imj-prg.fr/~bernard.teissier/documents/Overweightfinal.pdf
 \bibitem{T2} \bysame, \textit{The hunting of invariants in the geometry of discriminants}, in: Real and complex singularities 
(Proc. Ninth Nordic Summer School NAVF Sympos. Math., Oslo, 1976), 565-678. Sijthoff and Noordhoff, Alphen aan den Rijn, 1977.
\bibitem{T3} \bysame, \textit{Appendix 3: Bifurcations and Gibbs phase rule},  in: Oeuvres Math\'ematiques de Ren\'e Thom, Vol. III, Documents Math\'ematiques 19, Soci\'et\'e Math\'ematique de France, 2022, 167-170.
\bibitem{T4} \bysame  \textit{Autour d'une question de Michel Herman}, preprint, available at  https://webusers.imj-prg.fr/~bernard.teissier/old-papers.html
\bibitem{T5} \bysame, \textit{Valuations, deformations, and toric geometry}, 
in: Valuation Theory and its applications, Vol. II, Fields Inst. Commun. 33, AMS., Providence, RI., 2003, 361-459.  Available at https://webusers.imj-prg.fr/~bernard.teissier/documents/NewValfinal.pdf
\bibitem{T6}\bysame, \textit{Approximating rational valuations by Abhyankar semivaluations},  in: Workshop on Asymptotic methods in commutative algebra, Oberwolfach Reports, Vol. 13, Issue 4  (2016) 3214-3217.
\bibitem{T7} \bysame, \textit{Cycles \'evanescents, sections planes, et conditions de Whitney}, in Singularit\'es \`a Carg\`ese, Ast\'erisque 7-8, SMF, 1973.
\bibitem{Tev1} J. Tevelev, On a question of B. Teissier. \textit{Collectanea Math.} \textbf{65} (1) (2014), 61-66. (Published on line February 2013. DOI 10.1007/s13348-013-0080-9)
\bibitem{Tev2} \bysame. \textit{Compactifications of subvarieties of tori},  American journal of mathematics 129 (4), 2007, 1087-1104.

 \end{thebibliography}
  \end{document}